\documentclass[letterpaper, 11pt]{amsart}

\usepackage{amsmath}
\usepackage{mathtools}
\usepackage[all]{xy}
\usepackage[latin1]{inputenc}
\usepackage{varioref}
\usepackage{amsfonts}
\usepackage{amssymb}
\usepackage{bbm}
\usepackage{a4wide}
\usepackage{graphicx}
\usepackage{mathrsfs}
\usepackage{mathpazo}
 \usepackage[hypertexnames=false,
pdftex,
 	pdfpagemode=UseNone,
 	breaklinks=true,
 	extension=pdf,
 	colorlinks=true,
 	linkcolor=blue,
 	citecolor=blue,
 	urlcolor=blue,
 ]{hyperref}

\newcommand{\sO}{{\mathcal O}}

\newcommand{\scrB}{{\mathscr B}}

\newcommand{\scrL}{{\mathscr L}}

\newcommand{\C}{{\mathbb C}}

\newcommand{\N}{{\mathbb N}}
\renewcommand{\P}{{\mathbb P}}

\newcommand{\isom}{{\ \cong\ }}

\renewcommand{\to}[1][]{\xrightarrow{\ #1\ }}

\newtheoremstyle{citing}
  {}
  {}
  {\itshape}
  {}
  {\bfseries}
  {\textbf{.}}
  {.5em}
  {\thmnote{#3}}

\theoremstyle{plain}

\newtheorem{Thm}[subsection]{Theorem}
\newtheorem{Prop}[subsection]{Proposition}

\newtheorem{Lem}[subsection]{Lemma}

\theoremstyle{definition}

\newtheorem{definition}[subsection]{Definition}

\numberwithin{equation}{section}

\theoremstyle{remark}

\title[Base manifolds for Lagrangian fibrations]{Base manifolds for Lagrangian fibrations on hyperk\"ahler manifolds}

\author{Daniel Greb}

\address{Daniel Greb\\Arbeitsgruppe Algebra/Topologie\\
Fakult\"at f\"ur Ma\-the\-matik\\Ruhr-Universit\"at Bochum\\
44780 Bochum\\ Germany}
\email{daniel.greb@rub.de}
\urladdr{http://www.rub.de/ffm/Lehrstuehle/Greb/index.html}

\author{Christian Lehn}
\address{Christian Lehn\\Institut de Recherche Math\'ematique
Avanc\'ee\\ Universit\'e de Strasbourg\\
7 rue Ren\'e Descartes\\67084 Strasbourg Cedex\\France}
\email{lehn@math.unistra.fr}

\begin{document}
\thispagestyle{empty}

\begin{abstract}
Let $f\colon~X \to B$ be a fibration from a hyperk\"ahler manifold to a complex space $B$. Assuming that $B$ is smooth, we show that $B \cong \mathbb{P}^n$. This generalises a theorem of J.-M.~Hwang to the K\"ahler case.
\end{abstract}

\maketitle

\setlength{\parindent}{1em}
\setcounter{tocdepth}{1}

\section{Introduction}\label{sec intro}

One of the most important and startling conjectures in the study of hyperk\"ahler manifolds $X$  says that the base space of any non-trivial fibration $X\to B$ is the complex projective space $\P^n$, where $n = \dim X /2$. We refer the reader to~\cite[21.4]{GHJ} for a discussion of this conjecture. Any such fibration is automatically Lagrangian with respect to the holomorphic symplectic form by works of Matsushita; see \cite{Matsus03} or Section~\ref{sect:prelim} for a summary of his results.

With the additional hypothesis that $X$ be projective and that $B$ be smooth, the conjecture is known to hold by work of Jun-Muk Hwang~\cite{Hwang}. In this short note, we remove Hwang's projectivity assumption on $X$, and prove the following result:
\begin{Thm}\label{thm main}
Let $X$ be a hyperk\"ahler manifold, and let  $f\colon~X\to B$ be a fibration onto a complex space $B$. If $B$ is smooth, then $B\isom \P^n$.
\end{Thm}

Our proof is a simple combination of fundamental results due to A.~Fujiki~\cite{Fuj83}, J.-M.~Hwang~\cite{Hwang}, D.~Matsushita~\cite{Matsus03, Matsus09}, and Y.-T.~Siu~\cite{SiuDeformationRigidityPn}: after noticing that the base manifold has to be projective, we pull back a very ample line bundle from $B$ to $X$ and use this line bundle to deform the given fibration to a sequence of projective ones. Then, we apply Hwang's theorem to these projective deformations and use global deformation rigidity of $\P^n$ to conclude the desired result for the central fibre.

\section{Preliminaries}\label{sect:prelim}
We start by fixing our notation and by recalling definitions of the basic objects investigated in this note.
\begin{definition}
A compact K\"ahler manifold is called \emph{hyperk\"ahler} or \emph{irreducible holomorphic symplectic} if it is simply-connected, and if $H^0\bigl(X, \Omega_X^2\bigr) = \C\sigma$, where $\sigma$ is everywhere non-degenerate. A \emph{fibration} on $X$ is a (proper) surjective holomorphic map $f\colon~X\to B$ with $f_*\sO_X=\sO_B$ from $X$ to a complex space $B$ with $0<\dim B < \dim X$. In particular, the base $B$ of a fibration is normal, and $f$ has connected fibres. A \emph{Lagrangian fibration} on $X$ is a fibration $f: X \to B$ such that every irreducible component of every fibre of $f$ is a Lagrangian subvariety with respect to the holomorphic symplectic form $\sigma$.
\end{definition}
To make this note more self-contained, we collect some known results concerning fibrations on hyperk\"ahler manifolds, which we will use in the subsequent proof, in the following proposition.
\begin{Prop}\label{prop:generalisedMatsushita}
 Let $X$ be a hyperk\"ahler manifold of dimension $2n$, and let $f\colon~X \to B$ be a fibration onto a normal complex space $B$. Then, $f$ is a Lagrangian fibration onto a normal projective variety. In particular, $B$ has dimension $n$.
\end{Prop}
\begin{proof}
As explained in \cite[Thm.~1 and footnote]{AmerikCampanaTori}, using results of Var\-ouchas \cite{Va86, Va89} and the fundamental results of Matsushita \cite{Matsus99a, Matsus01a, Matsus00}, one shows without any a priori assumption on the base of the fibration that $B$ is a normal K\"ahler space. Then, \cite[Thm.~2.1 and Thm.~3.1]{Matsus03} imply the claim. 
\end{proof}

\section{Proof of Theorem \ref{thm main}}\label{sec proof}
Let $X$ be a hyperk\"ahler manifold of dimension $2n$, and let $f\colon X \to B$ be a fibration onto a smooth complex space $B$.

Since $B$ is projective by Proposition~\ref{prop:generalisedMatsushita}, there exists a very ample line bundle on $B$. Let $L$ denote its pullback under $f$. Furthermore, let $\mathscr{X} \to (S, 0)$ be the (smooth) Kuranishi space of $X$; in particular, $\mathscr{X} \to (S, 0)$ is a smooth family of hyperk\"ahler manifolds.  By \cite[Thm.~1.1(1) and 1.1(2)]{Matsus09} there exists a smooth hypersurface $(S_L,0)$ in $(S,0)$, and a line bundle $\mathscr{L}$ on the pullback $\mathscr{X}_L = S_L \times_S \mathscr{X}$ of $\mathscr{X}$ to $S_L$ such that the restriction of $\mathscr{L}$ to the fibre over the reference point $0$ is isomorphic to $L$. We denote the natural projection $\mathscr{X}_L \to S_L$ by $p$, and we note that both $\mathscr{X}_L$ as well as $p$ are smooth. 
As usual we will take a representative of the germ $(S_L, 0)$ and shrink it if necessary (keeping the base point), usually without mentioning this explicitly.

By \cite[Thm.~1.1(3) and Cor.~1.2]{Matsus09} the pushforward $p_*\scrL$ is a vector bundle, the canonical map $p^*p_*\scrL \to \scrL$ is surjective, and thus gives rise to a morphism  $F\colon~\mathscr{X}_L \to \P_{S_L}(p_*\scrL^\vee)$  over $S_L$ that extends $f\colon~X\to B$ to the whole family $\mathscr{X}_L$. 

\begin{Lem}
 We have $b_2(X) \geq 4$. In particular,  $S_L$ has positive dimension.
\end{Lem}
\begin{proof}
 Let $\alpha$ be a K\"ahler class on $X$. Since $L$ is not ample, its Chern class is not a multiple of $\alpha$, hence $h^{1,1}(X) \geq 2$, implying the first claim. For the second claim, note that the dimension of the Kuranishi space $S$ of $X$ is $\dim H^1(X, T_X) = h^{1,1}(X) \geq 2$, and that $S_L$ is a hyperplane in $S$.
\end{proof}

So $\mathscr{X}_L\to S_L$ is a positive-dimensional smooth family of hyperk\"ahler manifolds. Restricting the family $\mathscr{X}$ to a general smooth embedded disk $\Delta \subset S_L$ through the origin, we obtain a commutative diagram
 \begin{equation}\label{restr_fibration}
\begin{gathered}
\xymatrix{
\mathscr{X}_\Delta \ar[r]^F\ar[rd]_p& \scrB_\Delta \ar[d]^\pi\\
& \Delta,\\
}
\end{gathered}
\end{equation}
where $p\colon~\mathscr{X}_\Delta \to \Delta$ is a smooth family of hyperk\"ahler manifolds with smooth total space, and $\mathscr{B}_\Delta$ is the scheme-theoretic image of the fibration induced by a sufficiently high tensor power of $\mathscr{L}|_{\mathscr{X}_\Delta}$. Note that $\pi\colon~\scrB_\Delta \to \Delta$ is a flat family with normal total space.
\begin{Lem}\label{lem:reduced}
 The scheme-theoretic fibre $(\scrB_\Delta)_0 = \pi^{-1}(0)$ is reduced, hence smooth.
\end{Lem}
\begin{proof}Since $\scrB_\Delta$ is normal, it is non-singular at general points of the central fibre. Since $p$ is a smooth morphism, it follows from diagram \eqref{restr_fibration} that $(\scrB_\Delta)_0$ is generically reduced. Let $t$ be a coordinate on $\Delta$. We note that $\pi^*t$ is not a zerodivisor in any of the local rings $\mathscr{O}_{\scrB_\Delta, b}$ of points $b \in \pi^{-1}(0)$. Consequently, as $\mathscr{B}_\Delta$ satisfies Serre's condition $S_2$, the scheme $(\scrB_\Delta)_0 = (\pi^{-1}(0), \mathscr{O}_{\scrB_\Delta}/\pi^*t \cdot\mathscr{O}_{\scrB_\Delta})$ does not have any embedded components, and is therefore reduced, cf.~\cite[p.~125]{Matsumura}. Hence, $(\scrB_\Delta)_0 = \bigl((\scrB_\Delta)_0)\bigr)_{\mathrm{red}} = B$, which is smooth by assumption.
\end{proof}
Lemma~\ref{lem:reduced} implies that $\pi\colon~\scrB_\Delta \to \Delta$ is flat with smooth central fibre, hence a smooth morphism. Moreover,  \cite[Thm.~4.8(2)]{Fuj83} implies that there exists a dense subset $T\subset \Delta$ such that $X_t:=p^{-1}(t)$ is projective for all $t\in T$, see also \cite[Prop.~26.6]{GHJ}. Therefore, by Hwang's theorem \cite{Hwang} the fibre $B_t:=\pi^{-1}(t)$ is isomorphic to $\P^n$ for all $t\in T$. Hence, we find a sequence of points $\{p_\nu\}_{\nu \in \N}$ in $T \subset \Delta$ such that $\lim_{\nu \to \infty}(p_\nu) = 0$, and such that $B_t \cong \mathbb{P}^n$. Hence, global deformation rigidity of $\P^n$, see~\cite[paragraph following the Main Theorem]{SiuDeformationRigidityPn} implies that the central fibre is likewise isomorphic to $\P^n$.  This concludes the proof of Theorem~\ref{thm main}.
\\

{\it{Acknowledgements.}} The authors are grateful to S\"onke Rollenske for helpful discussions, and to Tim Kirschner, whose questions led to significant improvements in the exposition of our arguments. Moreover, the authors would like to thank the anonymous referees for helpful comments and remarks. The first named author gratefully acknowledges support by the Baden-W\"urttemberg Stiftung through the ``Eliteprogramm f\"ur Postdoktorandinnen und Postdoktoranden'', as well as by the DFG-Research Training Group GK 1821 ``Cohomological methods in geometry''.  The second named author was supported by the LABEX IRMIA, Strasbourg.

\enlargethispage*{\baselineskip}

%
\end{document}